\documentclass{article}
\usepackage[utf8]{inputenc}

\usepackage{authblk}
\usepackage{bm}
\usepackage{amsthm}
\usepackage{amsmath}
\usepackage{amssymb}
\usepackage{mathrsfs}
\usepackage{multirow}
\usepackage{multicol}
\usepackage[maxbibnames=10,minbibnames=3,maxalphanames=7,giveninits, doi=false, url=false, isbn=false,eprint=false]{biblatex}
\usepackage{enumerate}
\usepackage{url}
\usepackage{here}
\usepackage{appendix}
\usepackage{algorithm,algorithmic}

\def\eps{\varepsilon} 
 
\def\bmc{{\bm c}}

\def\bms{{\bm s}} 
\def\bmr{{\bm r}}   
  
 \def\cE{\mathcal{E}}

 \def\cM{{\mathcal M}}   
  \def\cS{{\mathcal S}}

\def\cX{{\mathcal X}}   

\def\bbZ{\mathbb{Z}}

\def\-{\mathchar`-}

\DeclareMathOperator\supp{supp}

\newtheorem{thm}{Theorem}[section]

\newtheorem{lem}[thm]{Lemma}

\AtBeginDvi{}

\numberwithin{equation}{section}

\DeclareMathOperator{\row}{{\mathrm{row}}}
\DeclareMathOperator{\col}{{\mathrm{col}}}

\DeclareMathOperator{\Mult}{Mult}
\DeclareMathOperator{\mult}{mult}

\title{Identities that represent powers of positive integers using multinomial coefficients}
\author{Shoichi Kamada\thanks{This work was supported by JST CREST Grant Number
JPMJCR2113, Japan.}}

\affil{University of Tsukuba, Japan\\
\texttt{kamada.shoichi.ft@u.tsukuba.ac.jp}}
\date{}

\begin{document}

\maketitle

\begin{abstract}
    In this paper, we show combinatorial identities that represent powers of positive integers using multinomial coefficients, 
    which do not come from the multinomial theorem and the multinomial Vandermonde's convolution.
\end{abstract}

\section{Introduction}

\subsection{Background}

The multinomial theorem states that for a positive integer $m$ and variables $x_{1},\ldots,x_{s}$, it holds that
\begin{equation}\label{eq:multinomial_theorem}
    (x_{1}+\cdots+x_{s})^{m}=\sum_{r_{1},\ldots,r_{s}}\binom{m}{r_{1},\ldots,r_{s}}x_{1}^{r_{1}}\cdots x_{s}^{r_{s}}, 
\end{equation}
where $r_{1},\ldots,r_{s}$ in the sum are taken over all nonnegative integers, and a multinomial coefficient $\binom{m}{r_{1},\ldots,r_{s}}$ is defined by
\[
\binom{m}{r_{1},\ldots,r_{s}}:=
\begin{cases}
\frac{m!}{r_{1}!\cdots r_{s}!} &\text{if $r_{1}+\cdots+ r_{s}=m$,}\\
0 & \text{otherwise.}
\end{cases}
\]
Substituting $x_{1}=\cdots=x_{s}=1$ in \eqref{eq:multinomial_theorem}, 
we have 
\begin{equation}\label{eq:one_multinomial_coefficient_previous}
s^{m}=\sum_{r_{1},\ldots,r_{s}}\binom{m}{r_{1},\ldots,r_{s}}.
\end{equation}
Taking a look at the coefficient of $x^{r_{1}}\cdots x^{r_{s}}$ in 
\[
(x_{1}+\cdots+x_{s})^{m_{1}+m_{2}}=(x_{1}+\cdots+x_{s})^{m_{1}}(x_{1}+\cdots+x_{s})^{m_{2}},
\]
the multinomial Vandermonde's convolution holds, i.e. 
    \[
    \binom{m_{1}+m_{2}}{r_{1},\ldots,r_{s}}=\sum_{k_{1},\ldots,k_{s}}\binom{m_{1}}{k_{1},\ldots,k_{s}}\binom{m_{2}}{r_{1}-k_{1},\ldots,r_{s}-k_{s}}.
    \]
Taking a summation on $r_{1},\ldots,r_{s}$, it follows from the multinomial theorem and \eqref{eq:one_multinomial_coefficient_previous} that
\begin{equation}\label{eq:two_multinomial_coefficients_previous}
s^{m}=\sum_{r_{1},\ldots,r_{s},k_{1},\ldots,k_{s}}\binom{m_{1}}{k_{1},\ldots,k_{s}}\binom{m_{2}}{r_{1}-k_{1},\ldots,r_{s}-k_{s}}.
\end{equation}
From identities \eqref{eq:one_multinomial_coefficient_previous} and \eqref{eq:two_multinomial_coefficients_previous}, we found that a power of a positive integer $s^{m}$ can be represented as a sum of multinomial coefficients or a sum of products of two multinomial coefficients. 

Some literatures on the classical and multinomial Vandermonde's convolution can be found in 
\cite{concrete_mathematics_1994,Vandermonde_type_identity_2014, multinomial_vandermonde_convolution_permanent_2021}. 
Another type identity that represents a multinomial coefficient using multinomial coefficients can be found in \cite{multinomial_identity_2022}. 
Generalizations of classical and multinomial Vandermonde's convolution can be found in \cite{generalizations_vandermonde_convolution_Gould_1956} and \cite{q_multinomial_vandermonde_convolution_SULANKE_1981}, respectively.
\subsection{Contribution}
Our contribution is to give identities that represent $s^{m}$ as sums using multinomial coefficients, which do not come from the multinomial theorem.
Let $K$ denote a nonempty set. 
For positive integers $m$ and $n$, 
let $\cM_{m,n}(K)$ denote the set of $m$-by-$n$ matrices whose components are in $K$. 
For a positive integer $s$, let $V$ denote a subset of $K^{n}$ such that $|V|=s$, 
and let $\cX\subseteq \cM_{m,n}(K)$ denote the set
\[
\{{}^{t}({}^{t}\bmr_{1},\ldots,{}^{t}\bmr_{m})\in\cM_{m,n}(K)\colon \bmr_{i} \in V\}.
\]
Then we have the trivial identity
\begin{equation}\label{eq:trivial_identity}
s^{m}=|\cX|.
\end{equation}
We will define some equivalence relations $\sim$ and sets $\cX\big/{\sim}$ of equivalence classes $[A]$. 
Then it is clear that the following identity holds.
\begin{equation}\label{eq:basic_identity}
s^{m}=\sum_{[A]\in\cX\big/\sim}|[A]|.
\end{equation}
Since we may take $K=\{1,\ldots,s\}$
and $V=\{(i,\ldots,i) \colon i\in K\}\subseteq K^{n}$, 
the arbitrariness of positive integers $s$ and $m$ is guaranteed in \eqref{eq:basic_identity}.
Our contribution is to give concrete definitions of equivalence relations $\sim$ and write the cardinality $|[A]|$ 
using one (or two) multinomial coefficient(s), where our identities require a set $V$ to be symmetrical in a group theoretical sense.

\paragraph*{Notation and terminology:}

For a multiset $S$, let $\supp(S)$ denote the support of $S$, i.e.  $\supp(S)$ is the set consisting of elements in $S$. 
For an element $a$, 
let $\mult_{S}(a)$ denote the multiplicity of $a$ in $S$. 
Let $\Mult(S)$ denote the multiset $\{\mult_{S}(a)\colon a\in\supp(S)\}$ of cardinality $|\supp(S)|$.
For a positive integer~$m$ and nonnegative integers $r_{1},\ldots,r_{q}$, 
let $S=\{r_{1}\ldots,r_{q}\}$ be a multiset of cardinality~$q$. 
Then a multinomial coefficient is defined as
\[
\binom{m}{S}:=
\begin{cases}
    \frac{m!}{r_{1}!\cdots r_{q}!} & \text{if $r_{1}+\cdots+r_{q}=m$,}\\
    0 & \text{otherwise.}
\end{cases}
\]
For a positive integer $n$, let $\cS_{n}$ denote the symmetric group on $\{1,\ldots,n\}$. 
A finite set $V\subseteq K^{n}$ is $\cS_{n}$-invariant if $\cS_{n}V=V$, 
where $\cS_{n}V$ is defined as a set $\{(r_{\sigma(1)},\ldots,r_{\sigma(n)})\in K^{n}\colon (r_{1},\ldots,r_{n})\in V,\, \sigma\in \cS_{n}\}$.

\section{Identities using sums of multinomial coefficients}

For a matrix $A=(\bmc_{1},\ldots,\bmc_{n})$, let $\col(A)$ denote the multiset $\{\bmc_{1},\ldots,\bmc_{n}\}$.
For $m$-by-$n$ matrices $A, B\in \cX$, 
a relation $A\overset{\mathrm{c}}{\sim}B$ holds if $\col(A)=\col(B)$.
Let $\cX/{\overset{\mathrm{c}}{\sim}}$ denote the set of equivalence classes in $\cX$ with respect to the equivalence relation $\overset{\mathrm{c}}{\sim}$,  
and $[A]_{\mathrm{c}}$ denote an equivalence class in $\cX/{\overset{\mathrm{c}}{\sim}}$ for a matrix $A\in \cX$.

\begin{thm}\label{thm:col_identity}
    For any positive integers $m$ and $s$ and any nonempty set $K$, there are a positive integer $n$ and an $\cS_{n}$-invariant subset $V\subseteq K^{n}$ with $|V|=s$ such that
    \[
    s^{m}=\sum_{[A]_{\mathrm{c}}\in \cX/{\overset{\mathrm{c}}{\sim}}}\binom{n}{\Mult(\col(A))}.
    \]
\end{thm}

\begin{proof}[Proof of Theorem \ref{thm:col_identity}]
    For a matrix $A=(A_{i,j})\in \cX$, 
    it holds that
    \[
    \left|[A]_{\mathrm{c}}\right|=|\{(A_{i,\sigma(j)})\colon \sigma\in\cS_{n}\}|=\binom{n}{\Mult(\col(A))} 
    \]
    for each $A\in \cX$, 
    where the argument of the cardinality $|\cdot|$ in the second expression above is a set. 
    Since we can consider this for any $[A]_{\mathrm{c}}\in \cX/{\overset{\mathrm{c}}{\sim}}$, we complete the proof.
\end{proof}

For a vector $\bmr=(r_{1},\ldots,r_{n})$, let $f(\bmr)$ denote a multiset $\{r_{1},\ldots,r_{n}\}$ of cardinality $n$, i.e. $f(\bmr)=\{r_{1},\ldots,r_{n}\}$.
For a matrix $A={}^{t}({}^{t}\bmr_{1},\ldots,{}^{t}\bmr_{m})$, 
let $\row_{f}(A)$ denote the multiset $\{f(\bmr_{1}),\ldots, f(\bmr_{m})\}$. 
For $m$-by-$n$ matrices $A,B\in\cX$, 
a relation $A\overset{\mathrm{r}}{\sim}B$ holds if $\row_{f}(A)=\row_{f}(B)$.
Let $\cX/{\overset{\mathrm{r}}{\sim}}$ denote the set of equivalence classes in $\cX$ with respect to the equivalence relation $\overset{\mathrm{r}}{\sim}$,  
and $[A]_{\mathrm{r}}$ denote an equivalence class in $\cX/{\overset{\mathrm{r}}{\sim}}$ for a matrix $A\in \cX$. 

\begin{thm}\label{thm:row_f_identity}
    For any positive integers $m$ and $s$ and any nonempty set $K$, there are a positive integer $n$ and an $\cS_{n}$-invariant subset $V\subseteq K^{n}$ with $|V|=s$ such that
    \[
    s^{m}=\sum_{[A]_{\mathrm{r}}\in \cX/{\overset{\mathrm{r}}{\sim}}}\binom{m}{\Mult(\row_{f}(A))}.
    \]
\end{thm}
\begin{proof}[Proof of Theorem \ref{thm:row_f_identity}]
    Let $A={}^{t}({}^{t}\bmr_{1},\ldots,{}^{t}\bmr_{m})\in\cX$.
    Put 
    \[
    \{S_{1},\ldots,S_{l}\}=\supp(\{f(\bmr_{1}),\ldots,f(\bmr_{m})\}).
    \]
    We consider a permutation $\sigma\in\cS_{m}$ that satisfies the following condition. 
    If for any $S\in \{S_{1},\ldots,S_{l}\}$, all rows $\bmr_{i_{1}},\ldots,\bmr_{i_{m'}}$ with $i_{1}<\cdots<i_{m'}$ satisfy $f(\bmr_{i_{j}})=S$ for all $j=1,\ldots,m'$, 
    then $\sigma(i_{1})<\cdots<\sigma(i_{m'})$, i.e. the order of $\bmr_{i_{1}},\ldots,\bmr_{i_{m'}}$ does not change in a matrix ${}^{t}({}^{t}\bmr_{\sigma(1)},\ldots,{}^{t}\bmr_{\sigma(m)})$.
    Since the cardinality $|[A]_{\mathrm{r}}|$ coincides with the number of such permutations $\sigma\in\cS_{m}$, which is equal to the number of rearrangements in a tuple $(f(\bmr_{1}),\ldots,f(\bmr_{m}))$, we have 
    \[
    |[A]_{\mathrm{r}}|=\binom{m}{\Mult(\row_{f}(A))}.
    \]
    Hence, we complete the proof.
\end{proof}
\section{An identity using a sum of products of two multinomial coefficients}
For $m$-by-$n$ matrices $A,B\in\cX$,  
a relation $A\overset{\mathrm{r,c}}{\sim}B$ holds if $A\overset{\mathrm{r}}{\sim}B$ and $A\overset{\mathrm{c}}{\sim}B$.
Let $\cX/{\overset{\mathrm{r,c}}{\sim}}$ denote the set of equivalence classes in $\cX$ with respect to the equivalence relation $\overset{\mathrm{r,c}}{\sim}$,  
and $[A]_{\mathrm{r,c}}$ denote an equivalence class in $\cX/{\overset{\mathrm{r,c}}{\sim}}$ for a matrix $A\in \cX$.

\begin{thm}\label{thm:row_col_identity}
For any positive integers $m$ and $s$ and any nonempty set $K$, there are a positive integer $n$ and an $\cS_{n}$-invariant subset $V\subseteq K^{n}$ with $|V|=s$ such that
    \[
    s^{m}=\sum_{[A]_{\mathrm{r,c}} \in \cX/{\overset{\mathrm{r,c}}{\sim}}}
    \binom{m}{\Mult(\row_{f}(A))}
    \binom{n}{\Mult(\col(A))}.
    \]
\end{thm}

\begin{lem}\label{lem:implication_between_equivalence_relations}
    Let $A,B\in\cX$. Then 
    $A\overset{\mathrm{c}}{\sim}B$ implies $A\overset{\mathrm{r,c}}{\sim}B$.
\end{lem}
\begin{proof}[Proof of Lemma \ref{lem:implication_between_equivalence_relations}]
Let $A=(A_{i,j})={}^{t}({}^{t}\bmr_{1},\ldots,{}^{t}\bmr_{m})\in\cX$ 
and $B=(B_{i,j})={}^{t}({}^{t}\bms_{1},\ldots,{}^{t}\bms_{m})\in\cX$. 
Then a relation $A\overset{\mathrm{c}}{\sim} B$ holds if and only if 
\begin{equation}\label{eq:f_equality}
(f(\bmr_{1}),\ldots,f(\bmr_{m}))=(f(\bms_{1}),\ldots,f(\bms_{m}))
\end{equation}
and $(A_{i,j})=(B_{i,\tau(j)})$ for some $\tau\in\cS_{n}$.
Moreover, 
\eqref{eq:f_equality} implies $\row_{f}(A)=\row_{f}(B)$. Hence, we complete the proof of the lemma.
\end{proof}

\begin{proof}[Proof of Theorem \ref{thm:row_col_identity}]
Let $A=(A_{i,j})={}^{t}({}^{t}\bmr_{1},\ldots,{}^{t}\bmr_{m})\in\cX$. 
Basically, we need to determine whether 
\begin{equation}\label{eq:membership_r_c}
(A_{\sigma(i),\tau(j)})\in [A]_{\mathrm{r,c}}
\end{equation}
or not for all $\sigma\in\cS_{m}$ and all $\tau\in\cS_{n}$.
The cardinality $|[A]_{\mathrm{r,c}}|$ is the number of matrices $(A_{\sigma(i),\tau(j)})$ without repetition such that each pair $(\sigma,\tau)\in\cS_{m}\times\cS_{n}$ satisfies \eqref{eq:membership_r_c}. 
We show the validity of the definition of the equivalence relation $\overset{\mathrm{r,c}}{\sim}$ and 
\begin{equation}\label{eq:the_cardinality_r_c}
    |[A]_{\mathrm{r,c}}|=\binom{m}{\Mult(\row_{f}(A))}
    \binom{n}{\Mult(\col(A))}.
\end{equation}

\noindent
Put 
    \[
    \{S_{1},\ldots,S_{l}\}=\supp(\{f(\bmr_{1}),\ldots,f(\bmr_{m})\}).
    \]
For each $l'=1,\ldots,l$, we define 
\[
I_{l'}:=\{i\colon f(\bmr_{i})=S_{l'}\}
\]
and let $\sigma_{I_{l'}}$ be a permutation on $I_{l'}$. Then we define a permutation $\sigma\in\cS_{m}$ as 
\begin{equation}\label{eq:permutation_on_a_set_of_c_classes}
\sigma(i)=
\begin{cases}
    \sigma_{I_{l'}}(i) & \text{if $i \in I_{l'}$},\\
    i & \text{otherwise}.
\end{cases}
\end{equation}
For a permutation $\sigma\in\cS_{m}$ defined by \eqref{eq:permutation_on_a_set_of_c_classes}, 
a mapping $[(A_{i,j})]_{\mathrm{c}}\mapsto [(A_{\sigma(i),j})]_{\mathrm{c}}$ defines 
a permutation on the set $\cX/{\overset{\mathrm{c}}{\sim}}$. 
This statement also holds for any permutation $\sigma$ which is a composition of permutations defined in \eqref{eq:permutation_on_a_set_of_c_classes} for several $l'$. 
From Lemma \ref{lem:implication_between_equivalence_relations}, 
if two classes $[(A_{i,j})]_{\mathrm{c}}$ and $[(A_{\sigma(i),j})]_{\mathrm{c}}$ are different, then so are
two classes $[(A_{i,j})]_{\mathrm{r,c}}$ and $[(A_{\sigma(i),j})]_{\mathrm{r,c}}$.
Notice that the choice of a matrix $A\in\cX$ is arbitrary. 
To count all permutations $\sigma\in\cS_{m}$ up to $[(A_{i,j})]_{\mathrm{r,c}}=[(A_{\sigma(i),j})]_{\mathrm{r,c}}$, 
we may consider all permutations $\sigma\in\cS_{m}$ such that for any $S\in\{S_{1},\ldots,S_{l}\}$, 
\begin{equation}\label{eq:the_order_of_rows_permutation}
\sigma(i_{1})<\cdots<\sigma(i_{m'}) \quad\text{when $I_{l'}=\{i_{1}<\cdots<i_{m'}\}$.} 
\end{equation}
Notice that indices $i$ and $j$ of rows in a matrix $A={}^{t}({}^{t}\bmr_{1},\ldots,{}^{t}\bmr_{m})$ with $i<j$ does not necessarily imply $\sigma(i)<\sigma(j)$ for a permutation $\sigma\in\cS_{m}$ satisfying \eqref{eq:the_order_of_rows_permutation} when $f(\bmr_{i})=S$ and $f(\bmr_{j})=T$ for distinct $S,T\in\{S_{1},\ldots,S_{l}\}$
The number of permutations $\sigma\in\cS_{m}$ satisfying \eqref{eq:the_order_of_rows_permutation} coincides with the number of rearrangements in a tuple $(f(\bmr_{1}),\ldots,f(\bmr_{m}))$, 
which is equal to 
\[
\binom{m}{\Mult(\row_{f}(A))}.
\]

\noindent
Hence, we have
\begin{align*}
|[A]_{\mathrm{r,c}}|=&\left|\{(f(\bmr_{\sigma(1)}'),\ldots,f(\bmr_{\sigma(m)}'))\colon {}^{t}({}^{t}\bmr_{1}',\ldots,{}^{t}\bmr_{m}')\in [A]_{\mathrm{c}},\sigma\in\cS_{m}\}\right|\\
=&\left|[A]_{\mathrm{r}}\right| \left|[A]_{\mathrm{c}}\right|\\
=&\binom{m}{\Mult(\row_{f}(A))}\binom{n}{\Mult(\col(A))}
\end{align*}
for each matrix $A\in\cX$.

\end{proof}

\section{An identity where $K$ is a group}
In this section, let $K$ denote a group, which is multiplicatively written. For a vector $\bmr=(r_{1},\ldots,r_{n})\in K^{n}$, 
let $\bmr^{-1}$ denote $(r_{1}^{-1},\ldots,r_{n}^{-1})$ and let $\bmr^{1}$ denote $(r_{1},\ldots,r_{n})$.
For a matrix $A={}^{t}({}^{t}\bmr_{1},\ldots,{}^{t}\bmr_{m})$, 
let $t(A)$ denote the cardinality of the set $\{i \colon f(\bmr_{i})\not=f(\bmr_{i}^{-1})\}$. 
For vectors $\bmr_{1},\bmr_{2}\in K^{n}$, 
the relation $\bmr_{1}\overset{\mathrm{s}}{\sim} \bmr_{2}$ holds if $f(\bmr_{1})=f(\bmr_{2})$ or $f(\bmr_{1})=f(\bmr_{2}^{-1})$. 
For matrices 
$A={}^{t}({}^{t}\bmr_{1},\ldots,{}^{t}\bmr_{m})\in\cX$ 
and 
$B={}^{t}({}^{t}\bms_{1},\ldots,{}^{t}\bms_{m})\in\cX$, 
a relation $A\overset{\mathrm{rs}}{\sim}B$ holds if 
there exists a permutation $\sigma\in\cS_{m}$ such that
$\bmr_{i}\overset{\mathrm{s}}{\sim} \bms_{\sigma(i)}$ for all $i=1,\ldots,m$.
For matrices $A,B\in\cX$, 
a relation $A\overset{\mathrm{rs,c}}{\sim}B$ holds if 
$A\overset{\mathrm{rs}}{\sim}B$ and $A\overset{\mathrm{c}}{\sim}B$.
Let $\cX\big/{\overset{\mathrm{rs,c}}{\sim}}$ denote the set of equivalence classes in $\cX$ with respect to the equivalence relation $\overset{\mathrm{rs,c}}{\sim}$,  
and $[A]_{\mathrm{rs,c}}$ denote an equivalence class in $\cX\big/{\overset{\mathrm{rs,c}}{\sim}}$ for a matrix $A\in \cX$.

\begin{thm}\label{thm:row_sym_col_identity}
For any positive integers $m$ and $s$ and any group $K$, there are a positive integer $n$ and an $\cS_{n}$-invariant subset $V\subseteq K^{n}$ with $|V|=s$ such that
    \[
    s^{m}=\sum_{[A]_{\mathrm{rs,c}} \in \cX\big/{\overset{\mathrm{rs,c}}{\sim}}}
    2^{t(A)}
    \binom{m}{\Mult(\row_{f}(A))}
    \binom{n}{\Mult(\col(A))},
    \]
\end{thm}

\begin{lem}\label{lem:implication_between_equivalence_relations_r_c_to_rs_c}
    Let $A,B\in\cX$. Then 
    $A\overset{\mathrm{r,c}}{\sim}B$ implies $A\overset{\mathrm{rs,c}}{\sim}B$.
\end{lem}

\begin{proof}[Proof of Lemma \ref{lem:implication_between_equivalence_relations_r_c_to_rs_c}]
Let $A=(A_{i,j})={}^{t}({}^{t}\bmr_{1},\ldots,{}^{t}\bmr_{m})\in\cX$ 
and $B=(B_{i,j})={}^{t}({}^{t}\bms_{1},\ldots,{}^{t}\bms_{m})\in\cX$. 
Then $A\overset{\mathrm{r,c}}{\sim}B$ holds if and only if $\row_{f}(A)=\row_{f}(B)$  and $\col(A)=\col(B)$. 
Moreover, $\row_{f}(A)=\row_{f}(B)$ holds if and only if there exists a permutation $\sigma\in\cS_{m}$ such that $f(\bmr_{i})=f(\bms_{\sigma(i)})$ holds for all $i=1,\ldots,m$.

$A\overset{\mathrm{rs}}{\sim}B$ holds if and only if 
there exists a permutation $\sigma\in\cS_{m}$ such that 
$f(\bmr_{i})=f(\bms_{\sigma(i)})$ or $f(\bmr_{i})=f(\bms_{\sigma(i)}^{-1})$ holds
for all $i=1,\ldots,m$.
Hence, $A\overset{\mathrm{r,c}}{\sim}B$ implies $A\overset{\mathrm{rs,c}}{\sim}B$. 
\end{proof}

\begin{proof}[Proof of Theorem \ref{thm:row_sym_col_identity}]
    From Lemma \ref{lem:implication_between_equivalence_relations_r_c_to_rs_c} and the definitions of $\overset{\mathrm{r,c}}{\sim}$ and $\overset{\mathrm{rs,c}}{\sim}$, 
    we have 
    \[
    [{}^{t}({}^{t}\bmr_{1},\ldots,{}^{t}\bmr_{m})]_{\mathrm{rs,c}}
    =\bigcup_{(\eps_{1},\ldots,\eps_{m})\in\{-1,1\}^{m}}[{}^{t}({}^{t}\bmr_{1}^{\eps_{1}},\ldots,{}^{t}\bmr_{m}^{\eps_{m}})]_{\mathrm{r,c}}.
    \]
    Notice that the above union does not always give a partition. To give a partition of $[{}^{t}({}^{t}\bmr_{1},\ldots,{}^{t}\bmr_{m})]_{\mathrm{rs,c}}$, 
    we may define some subset $\cE\subseteq\{-1,1\}^{m}$ such that 
    \[
    [{}^{t}({}^{t}\bmr_{1},\ldots,{}^{t}\bmr_{m})]_{\mathrm{rs,c}}
    =\bigcup_{(\eps_{1},\ldots,\eps_{m})\in\cE}[{}^{t}({}^{t}\bmr_{1}^{\eps_{1}},\ldots,{}^{t}\bmr_{m}^{\eps_{m}})]_{\mathrm{r,c}}
    \]
    gives a partition, where 
    \[
    [{}^{t}({}^{t}\bmr_{1}^{\eps_{1}},\ldots,{}^{t}\bmr_{m}^{\eps_{m}})]_{\mathrm{r,c}}=[{}^{t}({}^{t}\bmr_{1}^{\eps_{1}'},\ldots,{}^{t}\bmr_{m}^{\eps_{m}'})]_{\mathrm{r,c}}
    \]
    with $(\eps_{1},\ldots,\eps_{m})$, $(\eps_{1}',\ldots,\eps_{m}')\in\cE$ implies 
    $(\eps_{1},\ldots,\eps_{m})=(\eps_{1}',\ldots,\eps_{m}')$.
\end{proof}


\section{Examples}

In previous sections, we gave four combinatorial identities that represent powers of positive integers.
In this section, we give examples of power of integers and the corresponding choices of sets $K$ and $V$. 

\paragraph*{Example 1}
Let $K=\{0,1\}$ and 
\[
V=\{(x_{1},\ldots,x_{n})\in K^{n}\colon |\{i\colon x_{i}=1\}|=h\}.
\]
Then the identities represent the power of the binomial coefficient $\binom{n}{h}^{m}$, 
where Theorem \ref{thm:row_sym_col_identity} requires a set $K=\{0,1\}$ to be an additive group (modulo $2$). 

\paragraph*{Example 2}
Let $K=\{0,1,\ldots,q-1\}$ be a $q$ element set. For all $a\in K$, let $h_{a}$ be nonnegative integers with $\sum_{a\in K}h_{a} = n$.   
Let 
\[
V=\{(x_{1},\ldots,x_{n})\in K^{n}\colon |\{i\colon x_{i}=a\}|=h_{a}\}.
\]
Then the identities represent the power of the multinomial coefficient $\binom{n}{\{h_{a}\colon a\in K\}}^{m}$, 
where Theorem \ref{thm:row_sym_col_identity} requires a set $K=\{0,1,\ldots,q-1\}$ to be an additive group (modulo $q$). 

\paragraph*{Example 3}
Let $K=\bbZ$ be the ring of integers. For a nonnegative integer $h$, let 
\[
V=\{(x_{1},\ldots,x_{n})\in K^{n}\colon \sum_{i=1}^{n}x_{i}^{2}=h\}
\]
and put $r_{n}(h)=|V|$, which is the number of representations of $h$ as sums of $n$ squares of integers.
Then the identities represent the power $r_{n}(h)^{m}$. 
\section{Concluding remarks}
If one gives a combinatorial identity, then it is natural to consider its weighted sum. 
In our cases, one can consider a situation to count matrices in $\cX$ of a prescribed rank. Then it is sufficient that $K$ is a ring or a field, and a weight function on matrices is defined to be $1$ if the rank of a matrix coincides a prescribed rank, $0$ otherwise.
On the other hand, since the classical Vandermonde's convolution gives rise to the probability distribution (called the hypergeometric distribution), so do our identities.
Especially, a sample space in a probability space is $\cX/{\sim}$, where $\sim$ is 
either $\overset{\mathrm{c}}{\sim}$, $\overset{\mathrm{r}}{\sim}$, 
$\overset{\mathrm{r,c}}{\sim}$ or $\overset{\mathrm{rs,c}}{\sim}$. 
If one considers a real-valued random variable on this probability space, 
then a weighted sum appears naturally in its expectation.

Recall that a matrix is a $2$-dimensional array. Hence, it is natural to extend a set $\cX$ of matrices to a set of tensors, i.e. $k$-dimensional arrays for a positive integer $k$. 
In the tensor setting, it is natural to make combinatorial identities. 
However, it will be complicated to consider permutations of indices of tensors.  We leave this problem for future work.
\printbibliography
\end{document}